\documentclass[11pt]{amsart}
\usepackage{geometry}               
\usepackage{graphicx}
\usepackage{tikz}
\usepackage{amssymb}
\usepackage{pdfsync}
\usepackage{hyperref}
\usepackage{verbatim} 
\usepackage{epstopdf}
\usepackage[english]{babel}

\newcommand{\pl}[1]{\foreignlanguage{polish}{#1}}
\renewcommand{\d}{\,{\rm d}}

\numberwithin{equation}{section}

\def\R{\mathbb{R}}

\def\N{\mathbb{N}}

\def\Co{\mathbb{C}}

\newtheorem{theorem}{Theorem}

\newtheorem{proposition}[theorem]{Proposition}
\newtheorem{lemma}[theorem]{Lemma}

\title[Dimension-free restriction inequalities]{Dimension-free Fourier restriction inequalities}

\author[D. Oliveira e Silva]{Diogo Oliveira e Silva}
\address[Diogo Oliveira e Silva]{ 
Center for Mathematical Analysis, Geometry and Dynamical Systems \&
Departamento de Matem\'{a}tica\\ 
Instituto Superior T\'{e}cnico\\
Av.\@ Rovisco Pais\\ 
1049-001 Lisboa, Portugal} 
\email{diogo.oliveira.e.silva@tecnico.ulisboa.pt}

\author[B. Wr{\'o}bel]{B{\l}a{\.z}ej Wr{\'o}bel}
\address[B{\l}a{\.z}ej Wr{\'o}bel]{
Institute of Mathematics
	of the Polish Academy of Sciences\\
	Śniadeckich 8\\
	00-656 \pl{Warsaw}\\
	Poland \& Institute of Mathematics\\
	University of Wrocław\\
	Plac Grun\-waldzki 2\\
	50-384 \pl{Wrocław}\\
	Poland}
\email{blazej.wrobel@math.uni.wroc.pl}


\begin{document}
\begin{abstract}
Let ${{\bf R}_{\mathbb S^{d-1}}}(p\to q)$ denote the best constant for the $L^p(\R^d)\to L^q(\mathbb S^{d-1})$ Fourier restriction inequality to the unit sphere $\mathbb S^{d-1}$, and let ${\bf R}_{\mathbb S^{d-1}} (p\to q;\textup{rad})$ denote the corresponding constant for radial functions. We investigate the asymptotic behavior of the operator norms ${{\bf R}_{\mathbb S^{d-1}}}(p\to q)$ and ${\bf R}_{\mathbb S^{d-1}} (p\to q;\textup{rad})$ as the dimension $d$ tends to infinity.  We further establish  a dimension-free endpoint Stein–Tomas inequality for
radial functions, together with the corresponding estimate for general functions which we
prove with an $O(d^{1/2})$ dependence.
Our methods rely on a uniform two-sided refinement of Stempak's asymptotic $L^p$ estimate of Bessel functions.
\end{abstract}

\maketitle

\section{Introduction}

The Fourier restriction problem to the unit sphere $\mathbb S^{d-1}\subset\R^{d}$ aims at finding exactly those triples $(p,q,d)$ for which  the inequality 
\begin{equation}\label{eq_restr}
    \|\widehat f\|_{L^q(\mathbb S^{d-1})} \leq {{\bf R}_{\mathbb S^{d-1}}} (p\to q) \|f\|_{L^p(\R^{d})}
\end{equation}
holds,  where ${\bf R}_{\mathbb S^{d-1}} (p\to q):=\sup_{\|f\|_p =1} \|\widehat f\|_{L^q(\mathbb S^{d-1})}$ denotes the best constant. The Fourier transform is given by 
\begin{equation}\label{eq_FT}
    \widehat f(\xi):= \int_{\R^d} f(x) e^{-2\pi i x\cdot\xi}\,\d x.
\end{equation}
Stein's restriction conjecture \cite[Problem 2]{St79} asserts that ${{\bf R}_{\mathbb S^{d-1}}} (p\to q)<\infty$  if and only if 
\begin{equation} 
\label{eq: recond}
1\leq p<\frac{2d}{d+1} \text{ and } q\leq\frac{d-1}{d+1}p'.\end{equation}
This has been established when $d=2$, see \cite{Fef} for the non-endpoint case $p'>3q$ and \cite{CS72, Zy74} for the endpoint $p'=3q$, but is open for all $d\geq 3$. 

\medskip

Inequality \eqref{eq_restr} can be taken over  radial functions $f=f(|\cdot|)\in L^p$ only:
\begin{align}
    \|\widehat f\|_{L^q(\mathbb S^{d-1})} \leq {{\bf R}_{\mathbb S^{d-1}}} &(p\to q;\textup{rad}) \|f\|_{L^p(\R^{d})},\text{ where }\label{eq_radrestr}\\
    {\bf R}_{\mathbb S^{d-1}} (p\to q;\textup{rad})&:=\sup_{\|f\|_p =1, f \text{ radial}} \|\widehat f\|_{L^q(\mathbb S^{d-1})}.\notag
\end{align}
It is an early observation of L.\@ Schwartz that  ${{\bf R}_{\mathbb S^{d-1}}} (p\to q;\textup{rad})<\infty$ if and only if $1\leq p<2d/(d+1)$; see \cite[\S11, Notes]{MS13} and the historical account in \cite[\S 5]{St76}.
 For general functions $f\in L^p$, the Stein--Tomas inequality \cite{St93,To75} states that 
 \[{{\bf R}_{\mathbb S^{d-1}}} (p\to 2)<\infty, \text{ for every } d\geq 2 \text{ and }1\leq p\leq p_\bullet:=2\frac{d+1}{d+3}.\] 
  The Stein--Tomas exponent $p_\bullet$ is endpoint in light of Knapp's construction \cite[Lemma 3]{St77}, which together with the previous observation regarding radial functions ensures that ${{\bf R}_{\mathbb S^{d-1}}} (p\to q)=\infty$ whenever  \eqref{eq: recond} does not hold. 
  \medskip

  We say that the {\it $(p,q)$ dimension-free restriction} holds  if ${\bf R}_{\mathbb S^{d-1}} (p\to q)\leq C_{p,q}$, where $C_{p,q}<\infty$ is independent of $d$. 
Our main result is conveniently  stated in terms of the following three Riesz regions of exponents\footnote{The point $(\frac1p,\frac1q)=(1,0)$ is special, as ${\bf R}_{\mathbb S^{d-1}} (1\to \infty)=1$ for every $d\geq 2$.}
 $(1/p,1/q)\in[0,1]^2$ which are depicted in Fig.\@ \ref{fig_dimfreeRest}:
\begin{align*}
    \textup{A}&:=\{(\tfrac1p,\tfrac1q): 1\leq p<2, 1\leq q<p'\};\\
    \textup{B}&:=\{(\tfrac1p,\tfrac1q): 1< p<2, p'\leq q\leq\infty\};\\
    \textup{C}&:=\{(\tfrac1p,\tfrac1q):2\leq p\leq \infty, 1\leq q\leq \infty\}.
\end{align*}
Since $p_\bullet\to 2$ and $(d-1)/(d+1)\to 1$ as $d\to\infty$, the above discussion (together with interpolation and Hölder's inequality) implies that: 
\begin{itemize}
    \item ${\bf R}_{\mathbb S^{d-1}} (p\to q)<\infty$ for sufficiently large $d$ if $(\frac1p,\frac1q)\in \textup{A};$ 
    \item ${\bf R}_{\mathbb S^{d-1}} (p\to q)=\infty$ for all $d\geq 2$, but ${{\bf R}_{\mathbb S^{d-1}}} (p\to q;\textup{rad})<\infty$ whenever $d$ is sufficiently large, if $(\frac1p,\frac1q)\in \textup{B};$ 
    \item ${{\bf R}_{\mathbb S^{d-1}}} (p\to q;\textup{rad})=\infty$, and therefore ${{\bf R}_{\mathbb S^{d-1}}} (p\to q)=\infty$, for all $d\geq 2$ if $(\frac1p,\frac1q)\in \textup{C}$. 
\end{itemize}
    Consequently, the quest for dimension-free restriction inequalities is only sensible within regions A and B. This is addressed in our first  main result.

    \begin{theorem}\label{thm_main}
        If $(\frac1p,\frac1q)\in \textup{A}$, then  $(p,q)$ dimension-free restriction holds in all sufficiently large dimensions $d$. More precisely, if $(\frac1p,\frac1q)\in \textup{A}$, then 
      \begin{equation}
      \label{eq_GenConv0}
      {\bf R}_{\mathbb S^{d-1}} (p\to q)\to 0,\text{ as }d\to\infty.
      \end{equation}
        Furthermore, if $(\frac1p,\frac1q)\in \textup{B}$ and $q=p'$, then 
         \begin{equation}\label{eq_RadConv0}
        {\bf R}_{\mathbb S^{d-1}} (p\to q;\textup{rad})\to 0, \text{ as }d\to\infty.
        \end{equation}
        On the other hand,  if  $(\frac1p,\frac1q)\in \textup{B}$ and $q>p'$, then 
        \begin{equation}\label{eq_RadDivInf}
            {\bf R}_{\mathbb S^{d-1}} (p\to q;\textup{rad})\to\infty,\text{ as }d\to\infty.
        \end{equation}    
    \end{theorem}

\begin{figure}[htbp]
  \centering

  \begin{tikzpicture} [scale = 5]
    \fill[black!20] (0, 0) rectangle (1, 1);
    \fill[black!30!green] (0.5, 1) -- (0.5, 0.5) -- (1, 0) -- (1, 1) -- cycle;
    \fill[white!40!red] (0.5, 0) -- (0.5, 0.5) -- (1, 0) -- cycle;
    \draw (0, 1) node[left] {$1$};
    \draw[dotted] (1, 1) -- (0, 0) node[below left] {$0$}; 
    \draw[dotted] (1, 0.5) -- (0, 0.5) node[left] {$\frac12$};
    \draw[dotted] (0.5, 1) -- (0.5, 0) node[below] {$\frac12$};
    \draw[very thick, yellow] (0.5, 0.5) -- (1, 0);
   \draw[thick] (0.5, 0) -- (0.5, 1);      \draw[thick] (0, 1) -- (0.5, 1);  
   \draw[thick] (0, 0) -- (0, 1); 
    \draw[very thick] (0, 0) -- (0.5, 0); \draw[thick, black!50!green] (0.5, 1) -- (1, 1);
      \draw[very thick, black!50!green] (1, 1) -- (1,0); 
      \draw[very thick, black!20!red] (1, 0) -- (0.5,0); 
    \draw[->] (1, 0) -- (1.05, 0) node[right] {$\frac1p$};
    \draw[->] (0, 0) -- (0, 1.05) node[above] {$\frac1q$};
    \draw[dotted] (1, 1) -- (1, 0) node[below] {$1$};
    \fill[black] (0.5, 0.5) circle [radius = 0.01];
    \fill[black] (0.5, 0) circle [radius = 0.01];
    \fill[black] (0.5, 1) circle [radius = 0.01];
    \draw[black] (1,0) circle [radius = 0.015];
        
        \draw[thick, black!60, text = black] (0.75, 0.7) .. controls (0.75, 0.7) .. (0.75, 0.7) node[below] {A};
   \draw[thick, black!60, text = black] (0.7, 0.2) .. controls (0.7, 0.2) .. (0.7, 0.2) node[left] {B};
      \draw[thick, black!60, text = black] (0.25, 0.5) .. controls (0.25, 0.5) .. (0.25, 0.5) node{C};
  \end{tikzpicture}
  
  \caption{Riesz diagram for the restriction problem to $\mathbb S^{d-1}$ as $d\to\infty$.}
  \label{fig_dimfreeRest}
\end{figure}
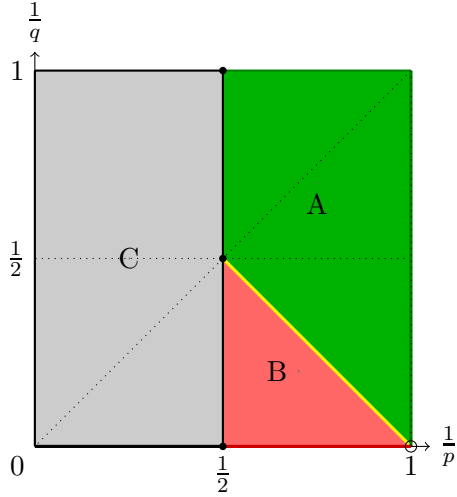

    The proof of Theorem \ref{thm_main} hinges on a dimension-free endpoint Stein--Tomas inequality for radial functions, together with the corresponding estimate for general functions which we prove with an $O(d^{1/2})$ dependence. 
    We present them as our second main result. 

    \begin{theorem}\label{thm_ST}
        Given $d\geq 2$, let $p_\bullet=2(d+1)/(d+3)$ denote the endpoint Stein--Tomas exponent. There exists an absolute constant $C<\infty$, independent of $d$, such that
        \begin{equation}\label{eq_WeakStein}
            {\bf R}_{\mathbb S^{d-1}} (p_\bullet\to 2)\leq Cd^{\frac12}.
        \end{equation}
        Moreover, there exist two absolute constants $0<a< A<\infty$, independent of $d$, such that
        \begin{equation}\label{eq_STradialfctns}
      a\le {\bf R}_{\mathbb S^{d-1}} (p_\bullet \to 2;\textup{rad})\leq A
        \end{equation}
    \end{theorem}
       Constant functions are conjectured to be global maximizers for  the endpoint adjoint Stein--Tomas inequality on $\mathbb S^{d-1}$ for every $d\geq 2$; see \cite[\S3.1]{FOS17}. Since ${\bf R}^\ast_{\mathbb S^{d-1}} (2\to p_\bullet')={\bf R}_{\mathbb S^{d-1}} (p_\bullet\to 2)$, this conjecture can be reformulated as
        \begin{equation}\label{eq_STConj}
           \frac{ \|\widehat\sigma\|_{L^{p_\bullet'}(\R^{d})}}{ \|{\bf 1}\|_{L^2(\mathbb S^{d-1})}} \stackrel{?}{=}{\bf R}_{\mathbb S^{d-1}} (p_\bullet\to 2) ,
        \end{equation}
        where $\sigma$ denotes the usual surface measure on $\mathbb S^{d-1}$, and its Fourier transform $\widehat{\sigma}$ is defined in \eqref{eq_sigmahat} below.
        So far, identity \eqref{eq_STConj} has only been confirmed on $\mathbb S^2$ by Foschi \cite{Fo15}.  Extensive evidence has been accumulated for $\mathbb S^1$; see the recent survey \cite[\S2]{NOST23} and the very recent works \cite{Be23,CG24}.
        In low dimensions $2\leq d\leq 60$, constant functions are known to be local maximizers \cite{GN22}. It is easy to verify that \eqref{eq_STConj} holds if ${\bf R}_{\mathbb S^{d-1}} (p_\bullet\to 2)$ is replaced by ${\bf R}_{\mathbb S^{d-1}} (p_\bullet \to 2;\textup{rad})$; see \eqref{eq_Holder}--\eqref{eq_Rrad} below.   Hence,   if conjecture \eqref{eq_STConj} does hold, then in light of \eqref{eq_STradialfctns} from Theorem \ref{thm_ST} we would have ${\bf R}_{\mathbb S^{d-1}} (p_\bullet \to 2)={\bf R}_{\mathbb S^{d-1}} (p_\bullet \to 2;\textup{rad})\in(a,A)$, which is a 
        dimension-free endpoint Stein--Tomas inequality for general functions. Thus estimates \eqref{eq_WeakStein}--\eqref{eq_STradialfctns} can  be regarded as partial progress  towards the validity of \eqref{eq_STConj} in general dimensions.

\medskip

The proof of Theorem \ref{thm_ST} naturally splits into two parts.  To justify \eqref{eq_WeakStein} we use  the analytic family
underlying the classical Stein–Tomas theorem, together with Hirschman’s form of complex
interpolation. The proof of inequality \eqref{eq_STradialfctns} relies on a two-sided uniform version of the asymptotic  $L^p$-estimate of Bessel functions from Stempak \cite[Eq.\@ (6)]{St00} which
%, modulo a twist which we detail in \S\ref{sec_42}--\S\ref{sec_43} below, 
will also be useful to establish \eqref{eq_RadConv0}--\eqref{eq_RadDivInf}. We consider it of independent interest and present it as a separate result. In what follows, by $A\lesssim B$ we mean that $A\le C B$ for an absolute constant $C$, and  analogously for $A\gtrsim B$.

\begin{proposition}\label{prop_Stempak} 
Given  $\nu\ge 2$
and $1\leq p\leq\infty$, let  $-\nu-\frac1p<\alpha < \frac12 - \frac1p$. Then the following uniform upper bound holds: 
\begin{equation}\label{eq_Stempak}
\begin{split}
&\left( \int_0^{\infty} \big| J_{\nu}(r)\,r^{\alpha}\big|^p\,\d r\right)^{\frac1p} \\
&\lesssim{2^{-\alpha}}\left\{ 
\begin{array}{ll}
\left(|p(\nu+\alpha)+1|^{-1}+(4-p)^{-1}+|p(\alpha - \frac12)+1|^{-1}\right)^{\frac1p} \nu^{\alpha - \frac12 + \frac1p}, & 1 \leq p < 4;\\
\left(|4(\nu+\alpha)+1|^{-1}+\log \nu+|4(\alpha - \frac12)+1|^{-1}\right)^{\frac14}\nu^{\alpha - \frac14}, & p=4;\\
\left(|p(\nu+\alpha)+1|^{-1}+(p-4)^{-1}+|p(\alpha - \frac12)+1|^{-1}\right)^{\frac1p}\nu^{\alpha - \frac13 + \frac{1}{3p}}, & 4 < p < \infty.
\end{array}
\right.
\end{split}
\end{equation}
Furthermore, for $p=\infty$ and $-\nu<\alpha<\frac12$, we have
\begin{equation}\label{eq_Stempak_inf}
\sup_{r\ge 0}|J_{\nu}(r)\,r^{\alpha}|\lesssim 2^{-\alpha} \nu^{\alpha - \frac13}.
\end{equation}
Moreover, under the same assumptions on $\nu,p,\alpha$ as before, the following uniform lower bound holds for every $1\le p< \infty:$
\begin{equation}\label{eq_Stempak_below}
 \left( \int_0^{\infty} \big| J_{\nu}(r)\,r^{\alpha}\big|^p\,\d r\right)^{\frac1p} \gtrsim 30^\alpha 
|p(\alpha - \tfrac12)+1|^{-\frac1p} \nu^{\alpha - \frac12 + \frac1p},
\end{equation}
while, for $p=\infty$ and $-\nu<\alpha<\frac12$, we have
\begin{equation}\label{eq_Stempak_below_inf}
 \sup_{r\ge 0}|J_{\nu}(r)\,r^{\alpha}| \gtrsim 30^\alpha 
 \nu^{\alpha - \frac12}.
\end{equation}
\end{proposition}

{\it A priori} stronger maximal and variational restriction estimates are known to follow from restriction inequalities like \eqref{eq_restr}, see \cite{Ko19}, and so the question of dimension-free bounds naturally arises in that setting as well. However, the argument in \cite{Ko19} leads to   estimates for  maximal and variational restriction operators which are qualitatively similar to Theorems \ref{thm_main}--\ref{thm_ST} in terms of the dimension. Because of this, and in order to emphasize the role of the restriction operator and of conjecture \eqref{eq_STConj}, we decided to phrase our results for the restriction operator only.
\medskip

The fact that the operator norms in \eqref{eq_GenConv0}--\eqref{eq_RadConv0} in Theorem \ref{thm_main} converge to $0$ when $d\to \infty$ may seem puzzling at first. However, a similar phenomenon occurs for the sharp Hausdorff--Young inequality,
\[
 \|\widehat f\|_{L^{p'}(\mathbb R^{d})} \leq {p^{\frac d{2p}}}{(p')^{-\frac d{2p'}}} \|f\|_{L^p(\R^{d})},\qquad 1\le p\le 2,
\]
proved by Babenko \cite{Ba61} whenever $p'$ is an even integer and by Beckner \cite{Beckner1} for general $p\in[1,2]$. In the latter range, if $p\notin\{1,2\},$ then $p^{1/p}(p')^{-1/p'}<1,$ and so  the sharp Babenko--Beckner constant, $(p^{1/p}(p')^{-1/p'})^{d/2}$, also tends to $0$ as $d\to \infty.$
The Hausdorff--Young inequality has been identified as a toy model for certain more elaborate phenomena involving the adjoint restriction, or {\it extension}, operator; see \cite{BBC09, BBCH09} and \cite[\S1.2]{COSS19}.
\medskip

        Questions about dimension-free estimates for various classes of operators constitute an important research theme in harmonic analysis. Even though the topic has been investigated for several decades,  many questions  remain open. The main classes of operators which have been studied in this context include Riesz transforms and Hardy--Littlewood maximal operators. The first instances of dimension-free estimates for these classes are due to E. M. Stein  \cite{SteinRiesz} together with J-O. Str\"omberg \cite{StStr}. The literature on this topic is vast. We refer the interested reader to  \cite[p. 487]{Hej1} for results related to Riesz transforms, and to \cite{BMSW4} and the references therein for results related to Hardy--Littlewood maximal functions.

%%%%%%%%%%%%%%%%%%
\subsection{Notation and preliminaries}
Given two nonnegative quantities
$A, B$, we continue to write $A \lesssim  B$ if
there exists an absolute constant $C>0$ such that $A\le CB$. We emphasize that,  in this notation, {\it $C$ is independent of the dimension $d$}. By $A \simeq B$, we mean that both $A\lesssim B$ and $B\lesssim A$ hold. 
Given $\Delta>0$, we also write $A \lesssim^{\Delta} B$ if there exists  an absolute constant $C>0$ such that
$A\le C^{\Delta}\, B.$ 
\medskip

The inverse of the Fourier transform $f\mapsto \widehat f$ given by \eqref{eq_FT} is 
\[f^\vee(x)=\int_{\R^d} f(\xi) e^{2\pi i \xi\cdot x}\,\d \xi.\]
It holds that $(f\ast g)^\wedge=\widehat f\cdot\widehat g$ and  $f=(\widehat f)^\vee$. Moreover, Plancherel's identity reads as follows:  $\|\widehat f\|_2=\|f^\vee\|_2=\|f\|_2$. 
\medskip

The usual surface measure $\sigma$ on the unit sphere $\mathbb S^{d-1}$ satisfies
\begin{equation}\label{eq_sphere}
   \sigma(\mathbb S^{d-1})=\frac{2\pi^{\frac d2}}{\Gamma(\frac d2)}=\frac{2\pi^{\nu+1}}{\nu!},
\end{equation}
where $\nu:=d/2-1$ and we abbreviate the Gamma function as $x!:=\Gamma(x+1)$. Of course, $\sigma(\mathbb S^{d-1})\to 0$ as $d\to\infty.$ The following effective version of Stirling's formula will be useful in the upcoming analysis.
\begin{lemma}[{  \cite[Eq.\ 6.1.38]{AS}}]
\label{lem_Stirling}
    For $x>0$,
   \[x!=\sqrt{2\pi}x^{x+\frac12}\exp(-x+\theta(x)),\]
where the function $\theta$ satisfies $0<\theta(x)<(12x)^{-1}.$
\end{lemma}
\noindent The Fourier transform of surface measure $\sigma$ on $\mathbb S^{d-1}$ is given by
\begin{equation}\label{eq_sigmahat}
    \widehat\sigma(\xi):=\int_{\mathbb S^{d-1}} e^{-2\pi i \omega\cdot\xi}\,\d\sigma(\omega)=2\pi |\xi|^{-\nu} J_{\nu}(2\pi |\xi|),
\end{equation}
where, as before,  $\nu=d/2-1$, and 
\begin{equation}\label{eq_Besseldef}
    J_\nu(r):=\sum_{n=0}^\infty \frac{(-1)^n}{n!\Gamma(n+\nu+1)}\left(\frac r2\right)^{2n+\nu}
\end{equation} defines the Bessel function for $r>0$;
see \cite[p.~347]{St93}.

\subsection{Outline of the paper}
In \S\ref{sec_Bessel}, we collect some known Bessel facts and use them to prove Proposition \ref{prop_Stempak}.
In \S\ref{sec_PT2} and \S\ref{sec_PT1}, we prove Theorems \ref{thm_ST}
 and \ref{thm_main}, respectively.

%%%%%%%%%%%%%%%%%%
\section{Bessel functions and the proof of Proposition \ref{prop_Stempak}}\label{sec_Bessel}
We start with some well-known pointwise uniform estimates for the Bessel function.
\begin{lemma}\label{lem_Besselnew}
The following uniform estimates hold:    \begin{equation}\label{eq_Bessel2}
 |J_\nu(r)|\leq\frac{r^\nu}{2^\nu\nu!} \qquad(\nu\geq 0,\,\,\, r>0)
  \end{equation}
 \begin{equation}\label{eq_Bessel3}
        |J_\nu(r)|\leq r^{-\frac12}\qquad(\nu\geq \frac12,\,\,\, r>2\nu)
    \end{equation}
    \begin{equation}\label{eq_Bessel4}
        |J_\nu(r)|\lesssim \nu^{-\frac14}(|r-\nu|+\nu^{\frac13})^{-\frac14} \qquad(\nu\geq 2,\,\,\, \frac{\nu}2<r<2\nu)
    \end{equation}
 \begin{equation}\label{eq_Bessel1}
 |J_\nu(r)|\lesssim \nu^{\frac16}r^{-\frac12}\qquad(\nu\geq 1,\,\,\, r>0)
 \end{equation}
\end{lemma}
\begin{proof}
Estimates \eqref{eq_Bessel2} and \eqref{eq_Bessel3} are contained in \cite[Lemmas 6 and 8]{COSS19}, respectively.
Estimate \eqref{eq_Bessel4} appears in \cite[Eq.\@ (4)]{St00} and follows from the estimates in \cite[p.~661]{BC89} via a straightforward case distinction.
Estimate \eqref{eq_Bessel1} follows from the more precise form in \cite[Theorem 2.1]{Ol06}. 
\end{proof}

The pointwise difference between $J_\nu(r)$ and the first term in its asymptotic expansion, $(2/(\pi r))^{1/2}\cos(r-(2\nu+1)\pi/4)$, is  of order $O(\nu^2 r^{-3/2})$; see 
\cite[Theorem 4]{Kr14} which comes with a sharp error term. Since we are interested in bounds that are uniform in $\nu$, the following result, which is a slight modification of  Krasikov's sharper approximation for the Bessel function from \cite[Theorem 5]{Kr14}, will be useful. Given $\nu>1/2$, let $\mu:=\nu^2-1/4$ and $\omega_\nu:=(2\nu+1)\pi/4$.
\begin{lemma}\label{lem_Krasikov}
     For every $\nu>1/2$ and $r>2\nu$, we have
 \begin{equation}\label{eq_Krasikov}
     J_\nu(r)=\sqrt{\frac{2}{\pi}}\frac{\cos(\mathcal B(r)-\omega_{\nu})}{(r^2-\mu)^{\frac14}}+g(r,\nu),
     \end{equation}
     where
     \begin{equation}\label{eq_Bg}
         \mathcal B(r):=(r^2-\mu)^{\frac12}+\sqrt\mu \arcsin \frac{\sqrt\mu}{r},\qquad\textrm{and}\qquad |g(r,\nu)|\leq  (r^2-\mu)^{-\frac34}.
     \end{equation}
\end{lemma}
\begin{proof}
    From \cite[Eq.\@ (9)]{Kr14}, we have that identity \eqref{eq_Krasikov} holds for every $r>\sqrt\mu$, as long as the function $\mathcal B$ is given by \eqref{eq_Bg}, and \[g(r,\nu)=\theta\frac{13\mu}{12\sqrt{2\pi}(r^2-\mu)^{\frac74}};\]
    here, $\theta$ is simply a real number whose absolute value does not exceed one. If $r>2\nu$, then we can further estimate:
    \[|g(r,\nu)|\leq\frac{13}{12{\sqrt{2\pi}}}\frac{\mu}{(r^2-\mu)^{\frac74}}\leq\frac{13}{12{\sqrt{2\pi}}}\frac{r^2/4}{(r^2-\mu)^{\frac74}}\leq\frac{13}{24{\sqrt{2\pi}}}\frac{r^2-\mu}{(r^2-\mu)^{\frac74}}\leq (r^2-\mu)^{-\frac34}.\qedhere\]
\end{proof}

We are now ready to prove Proposition \ref{prop_Stempak}.

\begin{proof}[Proof of Proposition \ref{prop_Stempak}]
Let $\nu\geq 2$, so that all estimates from Lemma \ref{lem_Besselnew} apply.  We consider first the case $1<p<\infty$, and decompose
\begin{align*}
\int_0^{\infty} \big| J_{\nu}(r)\,r^{\alpha}\big|^p\,\d r&= \int_0^{\frac{\nu}2} \big| J_{\nu}(r)\,r^{\alpha}\big|^p\,\d r+\int_{\frac{\nu}2}^{2\nu} \big| J_{\nu}(r)\,r^{\alpha}\big|^p\,\d r+\int_{2\nu}^{\infty} \big| J_{\nu}(r)\,r^{\alpha}\big|^p\,\d r\\
&:=I_1+I_2+I_3.
\end{align*}

We start with the proof of the upper bounds in \eqref{eq_Stempak}, estimating the integrals $I_j, j\in\{1,2,3\}$, from above. We  first treat $I_1$ and $I_3.$ In order to bound $I_1$, we use estimate \eqref{eq_Bessel2} followed by Stirling's formula as in Lemma \ref{lem_Stirling}, yielding
\begin{equation*}
\begin{split}
I_1=& \int_0^{\frac{\nu}2} \big| J_{\nu}(r)\,r^{\alpha}\big|^p\,\d r\le (2^\nu \nu!)^{-p}\int_0^{\frac{\nu}2} r^{p(\nu+\alpha)}\,\d r \lesssim 2^{-\nu p}  e^{\nu p} \nu^{-\nu p}\frac{(\frac{\nu}2)^{p(\nu+\alpha)+1}}{p(\nu+\alpha)+1} \\&\leq \frac{2^{-(2\nu+\alpha)p}e^{\nu p}\nu^{\alpha p +1}}{p(\nu+\alpha)+1}
= 2^{-\alpha p}\frac{(\frac e4)^{\nu p} \nu^{\alpha p +1}}{p(\nu+\alpha)+1}.
\end{split}
\end{equation*}
It follows that
\begin{equation}
\label{eq_I1}
I_1\lesssim \frac{2^{-\alpha p}}{|p(\nu+\alpha)+1|}\times \left\{ 
\begin{array}{ll}
 \nu^{p(\alpha - \frac12) + 1}, & 1 \leq p < 4;\\
\nu^{4\alpha - 1}, & p=4;\\
\nu^{p(\alpha - \frac13) + \frac{1}{3}}, & 4 < p \leq \infty.
\end{array}
\right.
\end{equation}
In order to bound $I_3$,   we use estimate \eqref{eq_Bessel3} and obtain
\begin{equation*}
    I_3\leq \int_{2\nu}^{\infty} r^{p(\alpha -\frac12)}\,\d r=\frac{(2\nu)^{p(\alpha -\frac12)+1}}{-p(\alpha -\frac12)-1}\leq \frac{\nu^{p(\alpha - \frac12) + 1}}{|p(\alpha - \tfrac12) + 1|},
\end{equation*}
from which it follows at once that
\begin{equation}
\label{eq_I3}
I_3\leq|p(\alpha - \tfrac12) + 1|^{-1}\times\left\{ 
\begin{array}{ll}
  \nu^{p(\alpha - \frac12) + 1}, & 1 \leq p < 4;\\
 \nu^{4\alpha - 1}, & p=4;\\
 \nu^{p(\alpha - \frac13) + \frac{1}{3}}, & 4 < p \leq \infty.
\end{array}
\right.
\end{equation}
In the case $p>4$, we used the simple inequality 
$p(\alpha - \frac12) + 1< p(\alpha - \frac13) + \frac{1}{3}$.
It remains to consider $I_2.$ Here, we use estimate \eqref{eq_Bessel4} and obtain
\begin{equation}
\label{eq_I2}
\begin{split}
    I_2&\lesssim \nu^{-\frac p4} \int_{\frac{\nu}2}^{2 \nu} (|r-\nu|+\nu^{\frac13})^{-\frac p4}r^{\alpha p}\,\d r \leq 2^{\alpha p} \nu^{-\frac p4+\alpha p} \int_{\frac{\nu}2}^{2 \nu} (|r-\nu|+\nu^{\frac13})^{-\frac p4}\,\d r\\
    &= 2^{\alpha p}\nu^{-\frac p4+\alpha p}\left(\int_{\frac{\nu}2}^{\nu} (-r+\nu+\nu^{\frac13})^{-\frac p4}\,\d r+\int_{\nu}^{2\nu} (r-\nu+\nu^{\frac13})^{-\frac p4} \,\d r \right).
    \end{split}
\end{equation}
The latter integrals can be explicitly computed.
If $1\le p<4$, then we note that $2^{\alpha p}<16^\alpha\lesssim 1$ since $\alpha<\frac14$, yielding
\begin{equation*}
I_2\lesssim \frac{\nu^{-\frac p4+\alpha p-\frac p4+1}}{4-p}=  \frac{\nu^{p(\alpha - \frac12) + 1}}{4-p}.
\end{equation*} If $p=4$, then we similarly have that $2^{\alpha p}\lesssim 1$, and thus  
\[
I_2\lesssim \nu^{4\alpha-1 }\left(\int_{\frac{\nu}2}^{\nu} (-r+\nu+\nu^{\frac13})^{-1}\,\d r+\int_{\nu}^{2\nu} (r-\nu+ \nu^{\frac13})^{-1}\,\d r \right)\lesssim  \nu^{4\alpha-1}\log\nu.
\]
Finally, when $p>4$, then  \eqref{eq_I2} boils down to
\[
I_2\lesssim \frac{2^{\alpha p}}{p-4}\nu^{-\frac p4+\alpha p-\frac p{12}+\frac13}= \frac{2^{\alpha p}}{p-4} \nu^{p(\alpha - \frac13) + \frac13}.
\]
Altogether we justified that
\begin{equation}
\label{eq_I2'}
I_2\lesssim \left\{ 
\begin{array}{ll}
 (4-p)^{-1}\nu^{p(\alpha - \frac12) + 1}, & 1 \leq p < 4;\\
\nu^{4\alpha - 1}\log\nu, & p=4;\\
2^{\alpha p} (p-4)^{-1} \nu^{p(\alpha - \frac13) + \frac{1}{3}}, & 4 < p \leq \infty.
\end{array}
\right.
\end{equation}
The estimates \eqref{eq_I1}--\eqref{eq_I3} and  \eqref{eq_I2'} together imply  the uniform upper bound  \eqref{eq_Stempak}.
\medskip

 To complete the proof of the upper bound, it remains to consider $p=\infty$ and justify \eqref{eq_Stempak_inf}. If $\nu$ and $\alpha$ are fixed numbers satisfying $-\nu<\alpha<1/2,$ then, for large enough $p$, we have
\[
 |p(\nu+\alpha)+1|^{-1}+(p-4)^{-1}+|p(\alpha - \frac12)+1|^{-1}\le 3,
\]
and $-\nu-\frac{1}{p}<\alpha<\frac{1}{2}-\frac{1}{p}.$
Consequently, \eqref{eq_Stempak_inf} follows from  letting $p\to \infty$ in \eqref{eq_Stempak}.

We now prove the lower bound \eqref{eq_Stempak_below}. It suffices to show that
\begin{equation}\label{eq_I3_below}
I_3\gtrsim^p \frac{(30)^{\alpha p}}{|p(\alpha - \tfrac12) + 1|}\nu^{p(\alpha - \frac12) + 1}.
\end{equation}
Appealing to Lemma \ref{lem_Krasikov}, we decompose $J_\nu(r)$ as in \eqref{eq_Krasikov}. Inequality \eqref{eq_I3_below} will follow once we justify that
\begin{equation}
\label{eq_goal_below}
\int_{10\nu}^{\infty}\left|{\sqrt{\frac2{\pi}}}\frac{|\cos(\mathcal B(r)-\omega_{\nu})|}{(r^2-\mu)^{\frac14}}-|g(r,\nu)|\right|^p r^{p\alpha }\,\d r\gtrsim^p \frac{(30)^{\alpha p}}{|p(\alpha - \frac12) + 1|}\nu^{p(\alpha - \frac12) + 1}.
\end{equation}
With that purpose in mind, we change variables $t=\mathcal B(r)-\omega_{\nu}$ whenever $r>10\nu$. In this case, we have  $r\leq \mathcal B(r)\leq \frac{6}{5}r$, and so $\frac45 r\leq \mathcal B(r)-\omega_\nu\leq \frac{6}{5}r$, that is, $t\simeq r$. Further note that 
\[
\mathcal B'(r)=\left(1-\frac{\mu}{r^2}\right)^{\frac12}\in\left(\tfrac9{10},1\right) \text{ if } r>10\nu.
\]
It follows that $\d t=\mathcal B'(r)\d r\simeq \d r.$ Denoting $r(t)=\mathcal B^{-1}(t+\omega_{\nu})$ we see that $r(t)\simeq t$ whenever $t>\mathcal B(10\nu)-\omega_{\nu}$ or, more precisely, $\frac56 t\le r(t)\le\frac54 t.$
As a consequence, we obtain
\begin{align*}
&\int_{10\nu}^{\infty}\left|{\sqrt{\frac2{\pi}}}\frac{|\cos(\mathcal B(r)-\omega_{\nu})|}{(r^2-\mu)^{\frac14}}-|g(r,\nu)|\right|^p r^{\alpha p}\,\d r\\
&\gtrsim^p \left(\frac 54\right)^{\alpha p}\int_{\mathcal B(10\nu)-\omega_{\nu}}^{\infty}\left|\sqrt{\frac2{\pi}}\frac{|\cos t|}{(r^2(t)-\mu)^{\frac14}}-|g(r(t),\nu)|\right|^p t^{\alpha p}\,\d t \\
& \gtrsim^p \left(\frac 54\right)^{\alpha p}\int_{12\nu}^{\infty}\left|\sqrt{\frac2{\pi}}\frac{|\cos t|}{(r^2(t)-\mu)^{\frac14}}-|g(r(t),\nu)|\right|^p t^{\alpha p}\,\d t.    
\end{align*}
Hence, our task reduces to proving that
\begin{equation}
\label{eq_I3_below'}
\int_{12\nu}^{\infty}\left|\sqrt{\frac2{\pi}}\frac{|\cos t|}{(r^2(t)-\mu)^{\frac14}}-|g(r(t),\nu)|\right|^p t^{\alpha p}\,\d t\gtrsim^p
\frac{(24)^{\alpha p}}{|p(\alpha - \frac12) + 1|}\nu^{p(\alpha - \frac12) + 1}.
\end{equation}
Given $k\in \N$, let $U_k:=[2\pi k,2\pi k+{\pi}/3)$. We have $|U_k|=\pi/3$ and $\cos t>1/2$ if $t\in U_k$. Thus, if 
$k\geq 1,$ then applying the estimate for $g$ from \eqref{eq_Bg} we have
  \begin{equation}
  \label{eq:cos-er bel}
  \sqrt{\frac2{\pi}}\frac{|\cos t|}{(r^2(t)-\mu)^{\frac14}}-|g(r(t),\nu)|\ge \frac{\frac14-(r^2(t)-\mu)^{-\frac12}}{(r^2(t)-\mu)^{\frac14}}\ge \frac{1}{2\sqrt{5t}} ,\qquad t\in U_k,
  \end{equation}
  where in the last inequality above we used the observation that $$25t^2/16\ge r^2(t)-\mu \ge 25t^2/36-\mu \ge t^2/2\ge 16,\qquad t> 12\nu.$$
  Since  $\alpha-\frac12<0$ we have
  $t^{p(\alpha-\frac12)}\geq  (2\pi (k+1))^{p(\alpha-\frac12)}$ for $t\in [2\pi k,2\pi (k+1))$. Thus using  \eqref{eq:cos-er bel} and disjointness of the $U_k$, we obtain
\begin{equation}
\label{eq:coi}
\begin{split}
&\int_{12\nu}^{\infty}\left|\sqrt{\frac{2}{\pi}}\frac{|\cos t|}{(r^2(t)-\mu)^{\frac14}}-|g(r(t),\nu)|\right|^p t^{\alpha p}\,\d t\\
&\gtrsim^p\sum_{k=\lceil\frac{8\nu}{\pi}\rceil}^\infty \int_{U_k}t^{p(\alpha-\frac12)}\,\d t  \gtrsim \sum_{k=\lceil\frac{8\nu}{\pi}\rceil}^\infty  (2\pi)\cdot\left(2\pi (k+1)\right)^{p(\alpha-\frac12)} \gtrsim   \sum_{k=\lceil\frac{8\nu}{\pi}\rceil}^\infty \int_{2\pi(k+1)}^{2\pi(k+2)} t^{p(\alpha-\frac12)}\,\d t\\
&=\int_{2\pi(\lceil\frac{8\nu}{\pi}\rceil+1)}^{\infty}t^{p(\alpha-\frac12)}\,\d t \ge \int_{16\nu+4\pi}^{\infty}t^{p(\alpha-\frac12)}\,\d t=\frac{(16\nu+4\pi)^{p(\alpha - \frac12) + 1}}{|p(\alpha - \tfrac12) + 1|}.
\end{split}
\end{equation}
Finally noting that $16\nu+4\pi \le 24 \nu,$ for $\nu\ge2,$ we see that $(16\nu+4\pi)^{p(\alpha - \frac12)+1}\gtrsim (24)^{p(\alpha-1/2)}\nu^{p(\alpha - \frac12)+1,}.$ In view of \eqref{eq:coi} this implies \eqref{eq_I3_below'} and hence  \eqref{eq_goal_below}, from which \eqref{eq_I3_below} follows. This completes the proof of the uniform lower bound \eqref{eq_Stempak_below}.

 It remains to justify \eqref{eq_Stempak_below_inf}. Note that, for fixed $\nu$ and $\alpha$ satisfying $-\nu<\alpha<1/2$ and large enough $p$, it holds
\[
|p(\alpha - \tfrac12)+1|=(1/2-\alpha)p-1\le (1/2-\alpha)p,
\]
and $-\nu-\frac1p<\alpha<\frac12-\frac1p.$
Hence, letting $p\to \infty$ in \eqref{eq_Stempak_below}, we reach \eqref{eq_Stempak_below_inf}.

This completes the proof of Proposition \ref{prop_Stempak}.
\end{proof}

%%%%%%%%%%%%%%%%%%
\section{Proof of Theorem \ref{thm_ST}}\label{sec_PT2}
The proof of Theorem \ref{thm_ST} naturally splits into the proofs of estimates \eqref{eq_WeakStein} and \eqref{eq_STradialfctns}.

\subsection{Proof of\eqref{eq_WeakStein}}
By duality, it suffices to prove that 
\begin{equation}\label{eq_goal17}
    \|f\ast\widehat\sigma\|_{L^{p_\bullet'}(\R^d)}\lesssim d \|f\|_{L^{p_\bullet}(\R^d)}.
\end{equation}
We will follow the classical analytic interpolation proof of the Stein--Tomas theorem. However, crucially for our purpose we will use the control on the constants coming from Hirschman's form of this interpolation argument.

Set
$A:=\frac{d+1}{2}$ and,
for $\Re z>-1$, define
$G_z(\xi):=\frac{(1-|\xi|^2)_+^z}{\Gamma(z+1)}.$
Continue this  to an entire family of tempered distributions. In particular, at $z=-1$ we have
$G_{-1}=\boldsymbol{\delta}(1-|\xi|^2)=\frac12\sigma$.
Define
$T_zf:=f*\widehat{G_z}$, and so
\begin{equation*}
%\label{eq_Tminus1}
T_{-1}f=\frac12 f*\widehat\sigma.
\end{equation*}
With our Fourier convention ($2\pi$ inside the exponential), the radial Fourier-transform formula  and Sonine's integral 
\cite[(10.22.19)]{DLMF}, initially for $\Re z>-1$ and then by
analytic continuation, give 
\begin{equation}\label{eq_Fourier-Gz}
\widehat{G_z}(x)
=
\pi^{-z}|x|^{-d/2-z}J_{d/2+z}(2\pi|x|).
\end{equation}
Next define the strip family
\[S_w:=T_{A(w-1)},\qquad0\leq\Re w\leq1.\]
Take
$\theta:=1-\frac1A=\frac{d-1}{d+1}$,
and so $A(\theta-1)=-1$, hence
$S_\theta=T_{-1}$.
Interpolation between $L^1\to L^\infty$ on $\Re w=0$ and $L^2\to L^2$ on $\Re w=1$ gives
\[
\frac1p=(1-\theta)+\frac{\theta}{2}
=\frac12+\frac1{2A}
=\frac{d+3}{2(d+1)},
\]
so $p=p_\bullet$. The only remaining point  is to keep track of the interpolation constant in $d$.

We start with estimating $S_w$ on the line  $w=1+iy$. Note that then one has $A(w-1)=iAy$, so that $S_w=T_{iAy}.$ Since $G_{iAy}$ is the Fourier multiplier of $S_{1+iy}$, Plancherel gives
\[
\|S_{1+iy}\|_{2\to2}
=
\|G_{iAy}\|_\infty
=
\frac1{|\Gamma(1+iAy)|}.
\]
By Euler's reflection formula and the recurrence relation for $\Gamma$, we have
$|\Gamma(1+it)|^{-2}
=\frac{\sinh(\pi|t|)}{\pi|t|}$,
where the quotient is interpreted as $1$ at $t=0$. It follows that
\begin{equation*}
%\label{eq_M1}
M_1(y):=\|S_{1+iy}\|_{2\to2}
=
\left(\frac{\sinh(\pi A|y|)}{\pi A|y|}\right)^{1/2},
\end{equation*}
and in particular $M_1(y)$ is an even function of $y.$ Estimating $\sinh u\le ue^u$ for $u\ge0$,
\begin{equation}\label{eq_M1-log}
\log M_1(y)\le \frac{\pi A}{2}|y|.
\end{equation}
%The dependence on $y$ must be retained; replacing \eqref{eq_M1-log} by a crude uniform vertical-line estimate destroys the desired polynomial dependence on $d$.

We now move to the line $w=iy.$ There we have $A(w-1)=-A+iAy,$ $S_w=T_{-A+iAy}$, and formula \eqref{eq_Fourier-Gz} becomes
\begin{equation}\label{eq_left-kernel}
\widehat{G_{-A+iAy}}(x)
=
\pi^{A-iAy}|x|^{1/2-iAy}
J_{-1/2+iAy}(2\pi|x|).
\end{equation}
Crucially, the real part of the Bessel order is now $-1/2$, independently of the dimension.
The estimate in \cite[Appendix~B.9, p.~584]{Grafakos},
combined with the recurrence and reflection formulas for $\Gamma$,
gives
\[
\sup_{r>0} r^{1/2}|J_{-1/2+it}(r)|
\le C\left(
\frac{1}{|\Gamma(1+it)|}+\frac{1}{|\Gamma(it)|}
\right)
\le C e^{c|t|},
\qquad t\in\mathbb R,
\]
for some absolute constants $C,c>0$, with the reciprocal gamma
functions interpreted by continuity at $t=0$.
Consequently,
$M_0(y):=\|S_{iy}\|_{L^1\to L^\infty}
\le C\pi^A e^{cA|y|}$,
and hence
\begin{equation}\label{eq_M0-log}
\log M_0(y)\le A\log\pi+C+cA|y|.
\end{equation}
Note that also $M_0(y)$ is even; this follows from the fact that the complex conjugate of \eqref{eq_left-kernel} equals $\widehat{G_{-A-iAy}}(x).$

Having prepared the boundary estimates \eqref{eq_M1-log} and \eqref{eq_M0-log} we are now ready to apply Stein's complex interpolation theorem.\footnote{For compactly supported simple functions $f,g$,  substituting the
power-series expansion of the Bessel function \eqref{eq_Besseldef}
 into \eqref{eq_Fourier-Gz}, and using the recurrence relation
and Stirling's formula for $\Gamma$,
\[
\bigl|\langle S_{u+iy}f,g\rangle\bigr|
\le C_{d,f,g}e^{C_d|y|},
\qquad 0\le u\le1,\quad y\in\mathbb R.
\]
Thus the family has admissible growth for Stein's interpolation
theorem.} Its appropriate form is the Hirschman refinement, see \cite{Hirschman} and \cite[Chapter V.4]{SW}.  For the strip $0\le \Re w\le1$, it yields
\begin{align}
\log \|S_\theta\|_{p_\bullet\to p_\bullet'}
\le \frac{\sin\pi\theta}{2}\int_\R
\Bigg[
\frac{\log M_0(y)}{\cosh\pi y-\cos\pi\theta}
+\frac{\log M_1(y)}{\cosh\pi y+\cos\pi\theta}
\Bigg]dy.
\label{eq_hirschman}
\end{align}
Set
$a:=\frac\pi A$.
Since $\theta=1-A^{-1}$,
we have
$\sin\pi\theta=\sin a$ and
$\cos\pi\theta=-\cos a$.
Since both $M_0$ and $M_1$ are even, \eqref{eq_hirschman} becomes
\begin{align}
\nonumber
\log \|S_\theta\|_{p_\bullet\to p_\bullet'}
&\le
\sin a\int_0^\infty
\frac{\log M_0(y)}{\cosh\pi y+\cos a}\,dy
+\sin a\int_0^\infty
\frac{\log M_1(y)}{\cosh\pi y-\cos a}\,dy\\
&:=L_A+R_A:=S_A. \label{eq: SA}
\end{align}
Therefore, the proof of \eqref{eq_goal17} (thus also of \eqref{eq_WeakStein}) reduces to the lemma below.

\begin{lemma}
    \label{lem: Es Hirch}
    Let $a=\frac\pi A$ and let $S_A$ be given by \eqref{eq: SA}. Then we have
    \[
    S_A\le \log A +C,
    \]
    where $C$ is a universal constant.
\end{lemma}
\begin{proof}
We first treat the term $L_A$ (coming from the left boundary $\Re w=0$). We have
\begin{equation*}
%\label{eq_harmonic-mass}
\sin a\int_0^\infty
\frac{dy}{\cosh\pi y+\cos a}
=
\frac a\pi
=
\frac1A,
\end{equation*}
and so the term $A\log\pi$ from \eqref{eq_M0-log} contributes exactly $\log\pi$. Moreover, $A\sin a\le\pi$, and for $A\ge3/2$ one has $0<a\le2\pi/3$, so
\[
\int_0^\infty
\frac{y\,dy}{\cosh\pi y+\cos a}\le C.
\]
Thus the linear term $c A|y|$ in \eqref{eq_M0-log} also contributes only an absolute constant. Hence
\begin{equation}\label{eq_left-O1}
\sin a\int_0^\infty
\frac{\log M_0(y)}{\cosh\pi y+\cos a}\,dy
\le C.
\end{equation}

We now estimate the term $R_A$ (coming from the right boundary $\Re w=1$). From \eqref{eq_M1-log}, we have
\begin{align*}
R_A
:=\sin a\int_0^\infty
\frac{\log M_1(y)}{\cosh\pi y-\cos a}\,dy 
\le
\frac{\pi A}{2}\sin a
\int_0^\infty
\frac{y\,dy}{\cosh\pi y-\cos a}.
\end{align*}
Changing variables $u=\pi y$,
\begin{equation*}
\label{eq_RA}
R_A
\le
\frac{A\sin a}{2\pi}
\int_0^\infty
\frac{u\,du}{\cosh u-\cos a}.
\end{equation*}
Set $b:=2\sin(a/2)$. Since
$\cosh u-\cos a
\ge \frac{u^2+b^2}{2}$,
we have
\[
\int_0^1\frac{u\,du}{\cosh u-\cos a}
\le \int_0^1\frac{2u\,du}{u^2+b^2}
=\log(1+b^{-2}).
\]
Also,
\[
\int_1^\infty\frac{u\,du}{\cosh u-\cos a}
\le \int_1^\infty\frac{u\,du}{\cosh u-1}
\leq C.
\]
As $A\ge 3/2$ and $a=\pi/A$, we have
$b\ge 2a/\pi=2/A$, and hence
\[
\int_0^\infty\frac{u\,du}{\cosh u-\cos a}
\le \log(1+A^2/4)+C
\le 2\log A+C.
\]
Finally, $A\sin a/(2\pi)=\sin a/(2a)\le 1/2$, and therefore
\begin{equation}\label{eq_right-logA}
{R_A\le \log A+C.}
\end{equation}
Thus, combining \eqref{eq_left-O1} and \eqref{eq_right-logA} we complete the proof of the lemma.
\end{proof}

\subsection{Proof of \eqref{eq_STradialfctns}}
The Fourier transform of a radial function $f:\R^d\to\Co$ is radial.
In this case, we abuse notation slightly and write $f(x)=f(r)$, where $x=r\omega\in\R^d$ for some $(r,\omega)\in (0,\infty)\times\mathbb S^{d-1}$. Identity \eqref{eq_sigmahat} then leads to
\begin{equation}\label{eq_rad1}
    \widehat f(\xi)=\int_0^\infty \left(\int_{\mathbb S^{d-1}} e^{-2\pi i \omega\cdot r\xi}\,\d\sigma(\omega)\right) f(r) r^{d-1}\,\d r
=2\pi |\xi|^{-\nu} \int_0^\infty J_\nu(2\pi r |\xi|) f(r) r^{\frac d2}\,\d r.
\end{equation}
The restriction of a radial function to $\mathbb S^{d-1}$ is constant. In particular, for every radial $f:\R^d\to\Co$ and exponent $q\in [1,\infty]$,
\begin{equation}\label{eq_rad2}
    \|\widehat f\|_{L^q(\mathbb S^{d-1})} = \sigma(\mathbb S^{d-1})^{\frac1q} \|\widehat f\|_{L^\infty(\mathbb S^{d-1})}.
\end{equation}
From \eqref{eq_rad1}--\eqref{eq_rad2} with $|\xi|=1$, it follows from Hölder's inequality that
\begin{multline}\label{eq_Holder}
    \frac{\|\widehat f\|_{L^q(\mathbb S^{d-1})}}{\sigma(\mathbb S^{d-1})^{\frac1q}} =2\pi \left|\int_0^\infty J_\nu(2\pi r) f(r) r^{\frac d2}\,\d r\right|\\
    \leq
\frac{2\pi}{\sigma(\mathbb S^{d-1})^{\frac 1p}} \left(\int_0^\infty \left| J_\nu(2\pi r) r^{\frac d2-\frac{d-1}p}\right|^{p'}\,\d r\right)^{\frac1{p'}} \|f\|_{L^p(\R^d)}.
\end{multline}
Equality in \eqref{eq_Holder} occurs if and only if $|f(r)|^p$ is a multiple of $|J_{\nu}(2\pi r) r^{1-d/2}|^{p'}$. Thus  
\begin{equation}\label{eq_Rrad}
    {\bf R}_{\mathbb S^{d-1}} (p  \to q;\textup{rad})=2\pi \sigma(\mathbb S^{d-1})^{\frac{1}{q}-\frac1{p}}\left(\int_0^\infty \left| J_\nu(2\pi r) r^{\frac d2-\frac{d-1}{p}}\right|^{p'}\,\d r\right)^{\frac1{p'}}.
\end{equation}
Identity \eqref{eq_Rrad} also follows from expanding the $L^{p'}$-norm of the extension operator acting on  the constant function.
At the endpoint Stein--Tomas exponent $(p,q)=(p_\bullet,2)$, the powers of $\sigma(\mathbb S^{d-1})$ from  \eqref{eq_Rrad} produce a total contribution of 
\begin{equation*}
%\label{eq_sph2}
    \sigma(\mathbb S^{d-1})^{\frac12-\frac1{p_\bullet}} = \left(\frac{2\pi^{\frac d2}}{\Gamma(\frac d2)}\right)^{-\frac1{d+1}}\lesssim \sqrt{d}, \text{ as } d\to\infty.
\end{equation*}
We now prepare to invoke Proposition \ref{prop_Stempak} with $p$ replaced by  $p_\bullet'=2(d+1)/(d-1)$ and $\alpha:=d/2-(d-1)/p_\bullet = \frac12(3-d)/(d+1).$ As  $d\to \infty$, we have $p_\bullet'\to 2$ and $\alpha\to -\frac12$ so that, for sufficiently large $d$,
\[
|p_\bullet'(\nu+\alpha)+1|^{-1}\lesssim 1,\qquad
(4-p_\bullet')^{-1}\simeq 1,\qquad |p_\bullet'(\alpha - \tfrac12)+1|^{-1}\simeq  1.
\]
Estimates \eqref{eq_Stempak}--\eqref{eq_Stempak_below} from Proposition \ref{prop_Stempak} then imply
\[(2\pi)^{\frac d2-\frac{d-1}{p_\bullet}}\left(\int_0^\infty \left| J_\nu(2\pi r) r^{\frac d2-\frac{d-1}{p_\bullet}}\right|^{p_\bullet'}\,\d r\right)^{\frac1{p_\bullet'}} \simeq \nu^{\frac d2-\frac{d-1}{p_\bullet}-\frac12+\frac1{p_\bullet'}}=\nu^{\frac{1-d}{2(d+1)}},\]
and consequently
\[\sigma(\mathbb S^{d-1})^{\frac12-\frac1{p_\bullet}}\left(\int_0^\infty \left| J_\nu(2\pi r) r^{\frac d2-\frac{d-1}{p_\bullet}}\right|^{p_\bullet'}\,\d r\right)^{\frac1{p_\bullet'}} \simeq \left(\frac{2\pi^{\frac d2}}{\Gamma(\frac d2)}\right)^{-\frac1{d+1}}(2\pi)^{\frac{d-3}{2(d+1)}}\nu^{\frac{1-d}{2(d+1)}}=: \ell_d.\]
One  checks that $\ell_d\to \sqrt{2/e}$, as $d\to\infty$, and this concludes the proof of \eqref{eq_STradialfctns} and therefore of Theorem \ref{thm_ST}.
%%%%%%%%%%%%%%%%%%
\section{Proof of Theorem \ref{thm_main}}\label{sec_PT1}

From the Riesz--Thorin convexity theorem, we have that 
\begin{equation}\label{eq_RT}
    {\bf R}_{\mathbb S^{d-1}} (p_\theta\to q_\theta) \leq {\bf R}_{\mathbb S^{d-1}} (p_0\to q_0)^{1-\theta}{\bf R}_{\mathbb S^{d-1}} (p_1\to q_1)^\theta,
\end{equation}
provided $\frac1{p_\theta}=\frac{1-\theta}{p_0}+\frac{\theta}{p_1}$ and $\frac1{q_\theta}=\frac{1-\theta}{q_0}+\frac{\theta}{q_1}$, for some $\theta\in[0,1]$.
Since ${\bf R}_{\mathbb S^{d-1}} (1\to\infty)=1$ and ${\bf R}_{\mathbb S^{d-1}} (1\to 1)\leq \sigma(\mathbb S^{d-1})$, inequality \eqref{eq_RT} implies 
\begin{equation}\label{eq_rightRiesz}
    {\bf R}_{\mathbb S^{d-1}} (1\to q)\leq \sigma(\mathbb S^{d-1})^{\frac1q},
\end{equation}
which of course relates to the fact that identity \eqref{eq_rad2} holds for radial functions.
\medskip

The proof of Theorem \ref{thm_main} naturally splits into the proofs of  \eqref{eq_GenConv0}, \eqref{eq_RadConv0} and \eqref{eq_RadDivInf}.

\subsection{Proof of \eqref{eq_GenConv0}}
Let $(\frac1p,\frac1q)\in \textup{A}$. In particular, $\frac1p+\frac1q>1$ and $p<2$. No generality is lost in assuming $p>1$, for otherwise \eqref{eq_rightRiesz}  together with the sentence after \eqref{eq_sphere} imply the desired conclusion.
Let 
\begin{equation}\label{eq_dSuffLarge}
d>\max\left\{4\left(\frac1{p-1}-1\right)^{-1}+1,\left(\frac1p+\frac1q-1\right)^{-1}-1\right\}.
\end{equation}
Note that the first condition  $d>4(\frac1{p-1}-1)^{-1}+1$ is equivalent to $p<p_\bullet$, where $p_\bullet=2(d+1)/(d+3)$ denotes the endpoint Stein--Tomas exponent as usual. From Theorem \ref{thm_ST}, we know that
\begin{equation}\label{eq_STbound2}
    {\bf R}_{\mathbb S^{d-1}} (p_\bullet\to 2)\leq Cd^{\frac12}.
\end{equation}
Define an auxiliary exponent $q(p)$ via
\begin{equation}\label{eq_defqp}
    \frac1{q(p)}:=1+\frac1{d+1}-\frac1p;
\end{equation}
see Fig.\@ \ref{fig_proof}.
From \eqref{eq_rightRiesz}, we then have that
\begin{equation}\label{eq_trivialIntBound}
      {\bf R}_{\mathbb S^{d-1}} (1\to q(1))\leq \sigma(\mathbb S^{d-1})^{\frac1{d+1}}.
\end{equation}

\begin{figure}[htbp]
  \centering

  \begin{tikzpicture} [scale = 5]
    \fill[black!20] (0, 0) rectangle (1, 1);
        \draw[->] (0, 1) -- (1, 0);
        \draw[->] (0.1, 1) -- (1, 0.1);
    \fill[white!60!green] (0, 1) -- (1, 0) -- (1, 1) -- cycle;
    \fill[white!60!red] (0, 0) -- (0, 1) -- (1, 0) -- cycle;
    \draw[very thick, yellow] (0, 1) -- (1, 0);
    \draw[very thick, gray] (0.5, 0.6) -- (0.5, 0.8) node[above, black] {$(\frac1p,\frac1q)$};
    \draw[very thick, gray] (0.5, 0.8) -- (0.5, 0.6) node[right, black] {$(\frac1p,\frac1{q(p)})$};
     \draw[very thick, gray] (1, .1) -- (.1, 1) node[above, black]{$(\frac1{p_\bullet},\frac12)$};
     \draw[very thick, gray] (.1, 1) -- (1, .1) node[right, black]{$(1,\frac1{q(1)})$};
    \draw[thick] (0, 0) -- (0.5, 0); \draw[thick, black!50!green] (0.5, 1) -- (1, 1);
      \draw[thick, black!20!red] (1, 0) -- (0,0); 
      \draw[thick, black] (0, 1) -- (0,0); 
      \draw[thick, black!50!green] (0, 1) -- (1,1); 
    \draw[->] (1, 0) -- (1.05, 0);
    \draw[->] (0, 0) -- (0, 1.05);
    \draw[thick, black!50!green] (1, 1) -- (1, 0);
    \fill[black] (0.1, 1) circle [radius = 0.01];
    \fill[black] (1, 0.1) circle [radius = 0.01];
    \fill[black] (0.5, 0.8) circle [radius = 0.01];
    \fill[black] (0.5, 0.6) circle [radius = 0.01];
    \draw[black] (1,0) circle [radius = 0.015];
  \end{tikzpicture}
  
  \caption{Proof of \eqref{eq_GenConv0} via interpolation.}
  \label{fig_proof}
\end{figure}
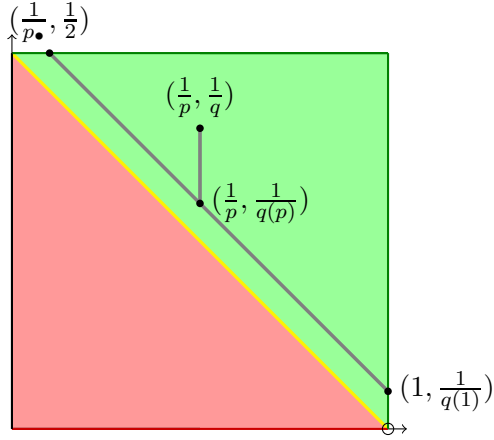

Note that $p\in(1,p_\bullet)$ and $q(p)\in(2,q(1))$.
Interpolating \eqref{eq_STbound2} and \eqref{eq_trivialIntBound} via \eqref{eq_RT}, 
\[{\bf R}_{\mathbb S^{d-1}} (p\to q(p))\leq {\bf R}_{\mathbb S^{d-1}} (p_\bullet\to 2)^{1-\theta}{\bf R}_{\mathbb S^{d-1}} (1\to q(1))^\theta
\leq (C\sqrt d)^{1-\theta} \sigma(\mathbb S^{d-1})^{\frac\theta{d+1}},\]
as long as $\theta\in (0,1]$ satisfies
\begin{equation}\label{eq_defTheta}
    \frac{\theta}{d+1}=\frac2{d-1}\left(\frac1p-\frac1{p_\bullet}\right).
\end{equation}
Note that  $\theta\in(0,1]$ is equivalent to $p\in[1,p_\bullet)$.
 Hölder's inequality then  yields
\begin{align}
    {\bf R}_{\mathbb S^{d-1}} (p\to q)\leq \sigma(\mathbb S^{d-1})^{\frac1q-\frac1{q(p)}}{\bf R}_{\mathbb S^{d-1}} (p\to q(p)) \leq
\sigma(\mathbb S^{d-1})^{\frac1q-\frac1{q(p)}}
(C\sqrt d)^{1-\theta} \sigma(\mathbb S^{d-1})^{\frac\theta{d+1}}\notag\\
=\sigma(\mathbb S^{d-1})^{\frac1p+\frac1q-1-\frac1{d+1}}\sigma(\mathbb S^{d-1})^{\frac2{d-1}(\frac1p-\frac1{p_\bullet})}
(C\sqrt d)^{1-2\frac{d+1}{d-1}(\frac1p-\frac1{p_\bullet})}.\label{eq_HolEst} 
\end{align}
To pass to the last line, we used identities \eqref{eq_defqp} and \eqref{eq_defTheta}. Note that in \eqref{eq_HolEst} both exponents on powers of $\sigma(\mathbb S^{d-1})$ are positive in light of \eqref{eq_dSuffLarge}, whereas the exponent on $C\sqrt d$ is nonnegative since $p\geq 1$, and strictly positive if  $p>1$.
From formula  \eqref{eq_sphere} and Lemma \ref{lem_Stirling}, it follows that the right-hand side of  \eqref{eq_HolEst} tends to 0 as $d\to\infty$, as long as $\frac1p+\frac1q>1$.

\subsection{Proof of \eqref{eq_RadConv0}}\label{sec_42}
As usual, let  $\nu=\frac d2-1.$ By \eqref{eq_sphere} and Lemma \ref{lem_Stirling}, we have 
\begin{equation}
\label{eq:aux1}
\sigma(\mathbb S^{d-1})\simeq (\pi e)^{\nu}\nu^{-(\nu+\frac12)},\text{ as }d\to\infty.
\end{equation}
Let $\alpha:=\frac d2-\frac{d-1}{p},$ and note that $\alpha=\nu(1-\frac2p)+\frac1{p'}.$ For large enough $d$, we then have 
\[(\nu+\alpha)p'+1=d>1,\text{ and }\] 
\begin{equation}\label{eq_awayfrom0}
    p'\left(\alpha-\tfrac12\right)+1=p'(d-1)\left(\tfrac12-\tfrac1p\right)+1<-1. 
\end{equation}
Thus estimate \eqref{eq_Stempak}  applied with $\alpha$ defined as above, and $p'\ge 2$ in place of $p$, yields
\begin{equation}\label{eq:aux2}
\left( \int_0^{\infty} \big| J_{\nu}(r)\,r^{\alpha}\big|^{p'}\,\d r\right)^{\frac1{p'}} \lesssim_{p'} 2^{-\alpha}\nu^{\alpha +\frac12}\le 2^{-\nu(1-\frac2p)+\frac1{p'}}\nu^{\nu(1-\frac2p)+1}
\end{equation}
Recalling \eqref{eq_Rrad} and using \eqref{eq:aux1}--\eqref{eq:aux2}, we then obtain
\[
 {\bf R}_{\mathbb S^{d-1}} (p  \to q;\textup{rad})\lesssim  (\pi e)^{\nu (\frac{1}{q}-\frac1{p})}\nu^{-(\nu+\frac12)(\frac{1}{q}-\frac1{p})}2^{-\nu(1-\frac2p)+\frac1{p'}}\nu^{\nu(1-\frac2p)+1}.
\]
Now, $q=p'$ implies  $\frac1q-\frac1p=1-\frac2p.$ Thus we may continue with
\[
 {\bf R}_{\mathbb S^{d-1}} (p  \to q;\textup{rad})\lesssim \nu^{2}  \left(\frac{\pi e}{2}\right)^{\nu(1-\frac2p)},
\]
which tends to $0$, as $d\to\infty$, because $1\le p<2$ and $\pi e>2.$

\subsection{Proof of \eqref{eq_RadDivInf}}\label{sec_43}
Let $\nu, \alpha$ be defined as in \S\ref{sec_42}. Again \eqref{eq_awayfrom0} holds for all sufficiently large $d,$ and so the lower bound \eqref{eq_Stempak_below} from Proposition \ref{prop_Stempak} implies
\begin{equation}\label{eq:aux3}
\left( \int_0^{\infty} \big| J_{\nu}(r)\,r^{\alpha}\big|^{p'}\,\d r\right)^{\frac1{p'}} \gtrsim 30^\alpha \nu^{\alpha -\frac12}\geq 30^{\nu(1-\frac2p)}\nu^{\nu(1-\frac2p)-\frac12}
\end{equation}
Now, using \eqref{eq_Rrad} together with \eqref{eq:aux1} and \eqref{eq:aux3} we obtain
\[
 {\bf R}_{\mathbb S^{d-1}} (p  \to q;\textup{rad})\gtrsim (\pi e)^{\nu (\frac{1}{q}-\frac1{p})}\nu^{-(\nu+\frac12)(\frac{1}{q}-\frac1{p})}30^{\nu(1-\frac2p)}\nu^{\nu(1-\frac2p)-\frac12}.
\]
Since $-1<\frac1q-\frac1p<1$ and $1-\frac2p-(\frac1q-\frac1p)=1-\frac1p-\frac1q$, we continue with
\[
 {\bf R}_{\mathbb S^{d-1}} (p  \to q;\textup{rad})\gtrsim \nu^{-1}(30 \pi e)^{-\nu}\nu^{\nu(1-\frac{1}{p}-\frac1{q})}.
\]
Finally, noting that $q>p'$ implies $1-\frac{1}{p}-\frac1{q}>0$, we obtain  ${\bf R}_{\mathbb S^{d-1}} (p  \to q;\textup{rad})\to\infty$ as $d\to \infty.$ 
This concludes the proof of Theorem \ref{thm_main}.

%%%%%%%%%%%%%%%%%%
\section*{Acknowledgements}
DOS  is partially supported by FCT/Portugal through project UIDB/04459/2020 with DOI identifier 10-54499/UIDP/04459/2020,
 by IST Santander Start Up Funds, and  by the Deutsche Forschungsgemeinschaft 
 (DFG, German Research Foundation) under Germany's Excellence Strategy – EXC-2047/1 – 390685813.
BW is supported by the National
Science Centre, Poland, grant Sonata Bis 2022/46/E/ST1/00036. For the purpose of Open Access the authors have applied a CC BY public copyright licence to any Author Accepted Manuscript (AAM) version arising from this submission.
The authors are grateful to Juan Antonio Barceló for providing a copy of his PhD thesis,  to Luz Roncal and Jim Wright for insightful discussions during the preparation of this work, and to the anonymous referee for valuable suggestions.

\end{document}